\documentclass[dvipdfm,11pt, reqno]{amsart}
\usepackage{amsmath,amsopn,amssymb,amsthm,multicol}
\usepackage[cp1251]{inputenc}
\usepackage[english]{babel}
\usepackage[dvips]{graphicx}
\usepackage[dvips]{hyperref}

\renewenvironment{proof}[1][Proof]{\textbf{#1.} }
{\ \rule{0.5em}{0.5em}}

\DeclareMathOperator{\ad}{ad}
\DeclareMathOperator{\Ad}{Ad}

\DeclareMathOperator{\trace}{trace}

\newtheorem{theorem}{Theorem}

\newtheorem{lem}{Lemma}

\newtheorem{remark}{Remark}
\newtheorem{ques}{Question}

\begin{document}

\title[The Ricci flow on generalized Wallach spaces]
{The Ricci flow on generalized Wallach spaces}
\author{N.A. Abiev}
\address{N.A. Abiev \newline
Taraz State University after M.Kh.~Dulaty, Taraz, Tole bi
str., 60, 080000, Kazakhstan}

\email{abievn@mail.ru}

\author{A. Arvanitoyeorgos}
\address{A. Arvanitoyeorgos \newline
University of Patras, Department of Mathematics, GR-26500 Rion,
Greece} \email{arvanito@math.upatras.gr}

\author[Yu.G.~Nikonorov]{Yu.G.~Nikonorov$^{*}$}

\address{Yu.G. Nikonorov \newline
South Mathematical Institute of  Vladikavkaz Scientific Centre of
the Russian Academy of Sciences, Vladikavkaz, Markus st. 22,
362027, Russia}
\email{nikonorov2006@mail.ru}

\author{P. Siasos}
\address{P. Siasos \newline
University of Patras, Department of Mathematics, GR-26500 Rion,
Greece} \email{petroblues@yahoo.gr}

\thanks{$^*$The project was supported in part by the
State Maintenance Program for the Leading Scientific Schools of
the Russian Federation (grant NSh-921.2012.1) and by Federal Target
Grant ``Scientific and educational personnel of innovative
Russia'' for 2009-2013 (agreement no. 8206, application no. 2012-1.1-12-000-1003-014).}

\begin{abstract} We consider the asymptotic behavior of the normalized Ricci flow
on generalized Wallach spaces that could be considered as special planar dynamical systems.
All non symmetric generalized Wallach spaces
can be naturally parametrized by three
positive numbers $a_1, a_2, a_3$.
Our interest is to determine the type of singularity of all singular points of the normalized Ricci flow
on all such spaces.
Our main result gives a qualitative answer
for almost all points $(a_1,a_2,a_3)$ in the cube $(0,1/2]\times (0,1/2]\times (0,1/2]$.

\vspace{2mm} \noindent Key word and phrases:
Riemannian metric, Einstein metric, generalized Wallach space, Ricci flow, Ricci curvature, planar dynamical system, real algebraic surface.

\vspace{2mm}

\end{abstract}

\maketitle

\section*{Introduction}

The study of the normalized Ricci flow equation
\begin{equation}
\label{ricciflow}
\dfrac {\partial}{\partial t} \bold{g}(t) = -2 \operatorname{Ric}_{\bold{g}}+ 2{\bold{g}(t)}\frac{S_{\bold{g}}}{n}
\end{equation}
for a $1$-parameter family of Riemannian metrics $\bold{g}(t)$ in a Riemannian manifold $M^n$    was
originally used by R. Hamilton in \cite{Ham} and since then it has attracted the interest of many mathematicians
(cf. \cite{ChowKnopf}, \cite{Topping}).
Recently, there is an increasing interest towards the study of the Ricci flow (normalized or not) on
homogeneous spaces and under various perspectives
(\cite{AnasChr}, \cite{Bo}, \cite{Bu}, \cite{GP}, \cite{Isenberg}, \cite{Lauret}, \cite{Payne}
and references therein).

The aim of the present work is to study the normalized Ricci flow for invariant Riemannian
metrics on generalized Wallach spaces.
These are compact homogeneous spaces $G/H$ whose isotropy representation
decomposes into a direct sum
$\mathfrak{p}=\mathfrak{p}_1\oplus \mathfrak{p}_2\oplus \mathfrak{p}_3$ of three
$\Ad(H)$-invariant irreducible modules satisfying
$[\mathfrak{p}_i,\mathfrak{p}_i]\subset \mathfrak{h}$ $(i\in\{1,2,3\})$ (\cite{Nikonorov2}, \cite{Lomshakov2}).
For a fixed bi-invariant inner product
$\langle\cdot, \cdot\rangle$ on the Lie algebra $\mathfrak{g}$ of the Lie group $G$,
any $G$-invariant Riemannian metric $\bold{g}$ on $G/H$ is determined by an $\Ad (H)$-invariant inner product
\begin{equation}
\label{metric}
(\cdot, \cdot)=\left.x_1\langle\cdot, \cdot\rangle\right|_{\mathfrak{p}_1}+
\left.x_2\langle\cdot, \cdot\rangle\right|_{\mathfrak{p}_2}+
\left.x_3\langle\cdot, \cdot\rangle\right|_{\mathfrak{p}_3},
\end{equation}
where $x_1,x_2,x_3$ are positive real numbers.
By using expressions for the Ricci tensor and the scalar curvature in \cite{Nikonorov2}
the normalized Ricci flow equation (\ref{ricciflow})
reduces to a system of ODE's of the form
\begin{equation}
\label{three_equat}
\dfrac {dx_1}{dt} = f(x_1,x_2,x_3), \quad
\dfrac {dx_2}{dt}=g(x_1,x_2,x_3), \quad
\dfrac {dx_3}{dt}=h(x_1,x_2,x_3),
\end {equation}
where $x_i=x_i(t)>0$ $(i=1,2,3)$, are parameters of the invariant metric (\ref{metric}) and
\begin{eqnarray*}
f(x_1,x_2,x_3)&=&
-1-\frac{A}{d_1}x_1 \left( \dfrac {x_1}{x_2x_3}-  \dfrac {x_2}{x_1x_3}-
\dfrac {x_3}{x_1x_2} \right)+2x_1 \frac{S_{\bold{g}}}{n},\\
g(x_1,x_2,x_3)&=&
-1-\frac{A}{d_2}x_2 \left( \dfrac {x_2}{x_1x_3}- \dfrac {x_3}{x_1x_2} -
\dfrac {x_1}{x_2x_3} \right)+2x_2 \frac{S_{\bold{g}}}{n},\\
h(x_1,x_2,x_3)&=&
-1-\frac{A}{d_3}x_3 \left( \dfrac {x_3}{x_1x_2}-  \dfrac {x_1}{x_2x_3}-
\dfrac {x_2}{x_1x_3} \right)+2x_3 \frac{S_{\bold{g}}}{n},\\
S_{\bold{g}}&=&
\frac{1}{2}\left( \dfrac {d_1}{x_1}+\dfrac {d_2}{x_2}+\dfrac {d_3}{x_3}-
A \left( \dfrac {x_1}{x_2x_3}+  \dfrac {x_2}{x_1x_3}+ \dfrac {x_3}{x_1x_2} \right) \right).
\end{eqnarray*}
Here $d_i$, $i=1,2,3$, are the dimensions of the corresponding irreducible modules
$\mathfrak{p}_i$, $n=d_1+d_2+d_3$ and $A$ is some special nonnegative number
(see Section \ref{Wallach}).
If $A\neq 0$, then by denoting $a_i:=A/d_i>0$, $i=1,2,3$,
the functions $f,g,h$ can be expressed in a more convenient form
(independent of $A$ and $d_i$) as
\begin{eqnarray*}
f(x_1,x_2,x_3)&=&-1-a_1x_1 \left( \dfrac {x_1}{x_2x_3}-  \dfrac {x_2}{x_1x_3}- \dfrac {x_3}{x_1x_2} \right)+x_1B,\\
g(x_1,x_2,x_3)&=&-1-a_2x_2 \left( \dfrac {x_2}{x_1x_3}- \dfrac {x_3}{x_1x_2} -  \dfrac {x_1}{x_2x_3} \right)+x_2B,\\
h(x_1,x_2,x_3)&=&-1-a_3x_3 \left( \dfrac {x_3}{x_1x_2}-  \dfrac {x_1}{x_2x_3}- \dfrac {x_2}{x_1x_3} \right)+x_3B,
\end{eqnarray*}
where
$$
B:=\left( \dfrac {1}{a_1x_1}+\dfrac {1}{a_2x_2}+\dfrac {1}{a_3x_3}- \left( \dfrac {x_1}{x_2x_3}+
\dfrac {x_2}{x_1x_3}+ \dfrac {x_3}{x_1x_2} \right) \right)
\left( \frac{1}{a_1} +\frac{1}{a_2}+ \frac{1}{a_3} \right)^{-1}.
$$

It is easy to check that the volume $V=x_1^{1/a_1}x_2^{1/a_2}x_3^{1/a_3}$ is a first integral of the system (\ref {three_equat}).
Therefore,  on the surface
\begin{equation}
\label{volume}
V\equiv 1
\end{equation}
we can reduce (\ref {three_equat}) to the system of two differential equations
of the type
\begin{equation}
\label{two_equat}
\dfrac {dx_1}{dt}=\widetilde f(x_1,x_2), \quad
\dfrac {dx_2}{dt}=\widetilde g(x_1,x_2),
\end{equation}
where
\begin{eqnarray*}
\widetilde f(x_1,x_2)&\equiv & f(x_1,x_2,\varphi(x_1,x_2)),\\
\widetilde g(x_1,x_2)&\equiv & g(x_1,x_2,\varphi(x_1,x_2)), \\
\varphi(x_1,x_2)&=&x_1^{ -\tfrac {a_3}{a_1}} x_2^{- \tfrac {a_3}{a_2}}.
\end{eqnarray*}

It is known (\cite{Nikonorov2}) that every generalized Wallach space admits at least one invariant Einstein metric.
Later in \cite{Lomshakov1}, \cite{Lomshakov2} a detailed study of invariant Einstein metrics was developed
for all generalized Wallach spaces. In particular, it was shown that there are at most four invariant
Einstein metrics (up to homothety)
for every such space.
It should be noted that invariant  Einstein metrics with $V=1$ correspond to singular points of
(\ref {two_equat}), therefore,
$(x_1^0,x_2^0,x_3^0)$ is a singular point of  the system (\ref {three_equat}), (\ref{volume})
if and only if  $(x_1^0,x_2^0)$ is a singular point of  (\ref {two_equat}).
It is our interest to determine the type of singularity of such points,
and our investigation concerns this problem for some special values of the parameters $a_1, a_2$, and $a_3$.
The main result in this direction is Theorem \ref{focus_center}, which gives a qualitative answer
for almost  all points (in measure theoretic sense) $(a_1,a_2,a_3)\in (0,1/2]\times(0,1/2]\times (0,1/2]$.
Note that the latter inclusion is fulfilled for any triple $(a_1,a_2,a_3)$ corresponding to some generalized Wallach
spaces (see the next section).
However we are interested in the behavior of the dynamical system (\ref{two_equat}) for all
values $a_i \in (0,1/2]$ despite the fact that some triples may not correspond to ``real'' generalized Wallach spaces.

We expect to give a more detailed study of the system (\ref{two_equat}) for various values of the parameters
$a_1,a_2,a_3$, which could help towards a deeper understanding of the behavior of the Ricci flow on more
general homogeneous spaces.
Also, it is quite possible that the system (\ref{two_equat}) is interesting not only for the parameters
$a_i\in(0,1/2]$, but as a more general dynamical system
other than the Ricci flow.
{\it It is clear that the system (\ref{three_equat}) is naturally defined for all
values of $a_1,a_2,a_3$ with
$a_1a_2+a_1a_3+a_2a_3\neq 0$, but for the system (\ref{two_equat}) we should assume
$a_1a_2a_3\neq 0$ additionally}.

\section{Generalized Wallach spaces}\label{Wallach}

We recall the definition and important properties of generalized Wallach spaces
(cf. \cite[pp. 6346--6347]{Nikonorov1} and \cite{Nikonorov2}).

Consider a~homogeneous almost effective compact space $G/H$ with  a (compact) semisimple connected Lie group
$G$ and its closed subgroup $H$. Denote by~$\mathfrak{g}$ and~$\mathfrak{h}$ the~Lie
algebras of~$G$ and~$H$ respectively. In~what follows,
$[\boldsymbol{\cdot}\,,\boldsymbol{\cdot}]$ stands for the~Lie bracket
of~$\mathfrak{g}$ and $B(\boldsymbol{\cdot}\,,\boldsymbol{\cdot})$ stands for
the~Killing form of~$\mathfrak{g}$. Note that $\langle\boldsymbol{\cdot}
\,,\boldsymbol{\cdot}\rangle=-B(\boldsymbol{\cdot}\,,\boldsymbol{\cdot})$
is a~bi-invariant inner product on~$\mathfrak{g}$.

Consider the~orthogonal complement~$\mathfrak{p}$ of~$\mathfrak{h}$ in~$\mathfrak{g}$ with respect
to~$\langle\boldsymbol{\cdot}\,,\boldsymbol{\cdot}\rangle$. Every
$G$-invariant Riemannian metric on~$G/H$ generates an~$\Ad(H)$-invariant
inner product on~$\mathfrak{p}$ and vice versa (\cite{Bes}). Therefore, it is possible to
identify invariant Riemannian metrics on~$G/H$ with $\Ad(H)$-invariant inner
products on~$\mathfrak{p}$ (if $H$ is connected then the property to be $\Ad(H)$-invariant is equivalent
to the property to be $\ad(\mathfrak{h})$-invariant). Note that the Riemannian metric generated by the~inner product
$\langle\boldsymbol{\cdot}\,,\boldsymbol{\cdot}\rangle\bigr\vert_{\mathfrak{p}}$ is called
{\it standard} or {\it Killing}.

Let~$G/H$ be  a homogeneous space such that its isotropy representation $\mathfrak{p}$
is decomposed as a direct sum of three
$\Ad(H)$-invariant irreducible modules pairwise orthogonal
with respect to~$\langle\boldsymbol{\cdot}\,,\boldsymbol{\cdot}\rangle$, i.e.
\begin{equation*}\label{decom1}
\mathfrak{p}=\mathfrak{p}_1\oplus \mathfrak{p}_2\oplus \mathfrak{p}_3,
\end{equation*}
\begin{equation*}\label{decom2}
\mbox{with   } [\mathfrak{p}_i,\mathfrak{p}_i]\subset \mathfrak{h} \mbox{ for } i\in\{1,2,3\}.
\end{equation*}
Since this condition on each
module resembles the~condition of local symmetry for
homogeneous spaces \big(a~locally symmetric homogeneous space~$G/H$ is
characterized by the~relation $[\mathfrak{p},\mathfrak{p}]\subset \mathfrak{h}$,
where $\mathfrak{g}=\mathfrak{h}\oplus \mathfrak{p}$
and~$\mathfrak{p}$ is $\Ad(H)$-invariant~\cite{Bes}\big), then spaces with this property were called
{\it three-locally-symmetric} in \cite{Lomshakov1, Lomshakov2}.
But in this paper we prefer the term {\it generalized Wallach spaces}, as in \cite{Nikonorov1}.

There are many examples of these spaces, e.g. the~flag manifolds
$$
SU(3)/T_{\max},
\ \ Sp(3)/Sp(1)\times Sp(1)\times Sp(1),
\ \ F_4/Spin(8).
$$
These spaces (known as {\it Wallach spaces}) are interesting because
they admit invariant Riemannian metrics of positive sectional curvature
(see~\cite{Wal}). The~invariant Einstein metrics on~$SU(3)/T_{\max}$
were classified in~\cite{Dat} and,
on the~remaining two spaces,
in~\cite{Rod}. In~each of the~cases, there exist exactly
four invariant Einstein metrics (up to proportionality). Other
classes of generalized Wallach spaces are the various K\"ahler $C$-spaces such as
$$
SU(n_1+n_2+n_3)\big/S\big(U(n_1)\times U(n_2)\times U(n_3)\big),
$$
$$
SO(2n)/U(1)\times U(n-1),
\quad
E_6/U(1)\times U(1)\times Spin(8).
$$
The~invariant Einstein metrics in the~above spaces were classified
in~\cite{Kim}. Each of these spaces admits four
invariant Einstein metrics (up to scalar), one of which is K\"ahler for an~appropriate
complex structure on~$G/H$. Another approach to
$SU(n_1+n_2+n_3)\big/S\big(U(n_1)\times U(n_2)\times U(n_3)\big)$ was used
in~\cite{Ar}. The~Lie group~$SU(2)$ \big($H=\{e\}$\big) is another example
of a~generalized Wallach space. Being 3-dimensional, this group
admits only one left-invariant Einstein metric which is a~metric of constant
curvature~(\cite{Bes}).

In~\cite{Nikonorov2},  it was shown that every generalized Wallach space
admits at least one invariant Einstein metric.
This result could not be improve in general (since e.g.~$SU(2)$ admits exactly one invariant Einstein metric).
Later in \cite{Lomshakov1}, \cite{Lomshakov2} a detailed study of invariant Einstein metrics was developed
for all generalized Wallach spaces. In particular, it is proved that there are at most four
Einstein metrics (up to homothety) for every such space.

Denote by $d_i$ the~dimension of~$\mathfrak{p}_i$. Let~$\big\{e^j_i\big\}$ be
an~orthonormal basis in~$\mathfrak{p}_i$ with respect to
$\langle\boldsymbol{\cdot}\,,\boldsymbol{\cdot}\rangle$, where
$i\in\{1,2,3\}$, $1\le j\le d_i=\dim(\mathfrak{p}_i)$. Consider
the~expression $[ijk]$ defined by
the~equality
$$
[ijk]=\sum_{\alpha,\beta,\gamma}
\left\langle\big[e_i^\alpha,e_j^\beta
            \big],e_k^\gamma
\right\rangle^2,
$$
where $\alpha$, $\beta$, and $\gamma$ range from~1 to ~$d_i$, $d_j$, and
$d_k$ respectively. The~symbols $[ijk]$ are symmetric in all three indices by bi-invariance
of the~metric
$\langle\boldsymbol{\cdot}\,,\boldsymbol{\cdot}\rangle$.
Moreover, for spaces under consideration, we have
$[ijk]=0$
if two indices coincide. Therefore, the~quantity
$A:=[123]$
plays an~important role.

By \cite[Lemma 1]{Nikonorov2} we get
$d_i \geq 2A$ for every $i=1,2,3$ with $d_i =2A$ if and only if $[\mathfrak{h},\mathfrak{p}_i]=0$.

Note that $A=0$ if and only if the space $G/H$ is locally a direct product of three compact irreducible symmetric spaces
(see  \cite[Theorem 2]{Lomshakov2}).

Suppose $A\neq 0$ and
let
\begin{equation}\label{ai}
a_i=A/d_i, \quad i\in \{1,2,3\}.
\end{equation}
It is clear that $a_i \in (0,1/2]$.
Note, that not every triple $(a_1,a_2,a_3)\in (0,1/2] \times (0,1/2] \times (0,1/2]$ corresponds to some generalized
Wallach spaces. For example, if $a_i=1/2$ for some $i$, then there is no generalized Wallach space
with $a_j\neq a_k$, where $i\neq j\neq k \neq i$ (see Lemma 4 in \cite{Nikonorov2}).
Moreover, every $a_i$  should be a rational number for a generalized
Wallach space with simple group $G$ (see (\ref{ai}), \cite[Lemma 1]{Lomshakov2} and
\cite[Table 1]{DZ}).

An explicit expression for the Ricci curvature of invariant metrics (\ref{metric})
is obtained in Lemma~2 of  \cite{Nikonorov2}. It immediately implies the system (\ref{three_equat}).

\section{Description of the singular points of the system (\ref{three_equat})}\label{sinpsec}

We will first give a description of the singular points of the system
(\ref{three_equat}), as Einstein metrics on generalized Wallach spaces.
An easy calculation shows that for $a_1a_2+a_1a_3+a_2a_3\neq 0$ the singular points $(x_1,x_2,x_3)$ of the
system (\ref{three_equat}) can be found from the equations
\begin{equation}\label{two_equat_sing1}
\begin{array}{l}
(a_2+a_3)(a_1x_2^2+a_1x_3^2-x_2x_3)+(a_2x_2+a_3x_3)x_1-(a_1a_2+a_1a_3+2a_2a_3)x_1^2=0,
\\
(a_1+a_3)(a_2x_1^2+a_2x_3^2-x_1x_3)+(a_1x_1+a_3x_3)x_2-(a_1a_2+2a_1a_3+a_2a_3)x_2^2=0.\\
\end{array}
\end{equation}
If $a_1a_2a_3\neq 0$ and $x_3=\varphi(x_1,x_2)$, then we also get singular points of the system (\ref{two_equat}).
Recall that we are interesting only for singular points with $x_i>0$, $i=1,2,3$.

Note that  system (\ref{two_equat_sing1}) is homogeneous (of degree 2) with respect to
$x_1,x_2,x_3$. It is easy to see that $x_i=x_j=0$ implies $x_k=0$ or $a_i(a_j+a_k)=a_j(a_i+a_k)=0$,
$i\neq j \neq k \neq i$. If we have a solution with $x_i=1$ and $x_j=0$, then we should have
$(4a_j^2-1)(a_i+a_k)(a_1a_2+a_1a_3+a_2a_3)=0$.
Therefore, if $a_i\in (0,1/2)$ for $i=1,2,3$, then  system (\ref{two_equat_sing1})
has no solution with zero component.
If $a_i\in (0,1/2]$, $i=1,2,3$ then it is proved in \cite{Lomshakov2} that
this system has (up to multiplication by a constant, for example, if we put $x_3=1$)
at least one and at most four solutions with positive components. A detailed information on these solutions can be found
in \cite{Lomshakov2}. We briefly review these results below.
\smallskip

{\it The case where at least two of $a_i$'s are equal.}
Without loss of generality we may assume that $a_1=a_2=b$ and
$a_3=c$.  Then  system (\ref{two_equat_sing1})
is equivalent to
the~following system:
\begin{equation}\label{solcase1}
\begin{array}{l}
(x_2-x_1)\big(x_3-2b(x_1+x_2)
         \big)=0,\\
x_2(x_3-x_1)+(b+c)(x^2_1-x^2_3)+(c-b)x_2^2=0.\\
\end{array}
\end{equation}

If~$x_2=x_1$ then the~second equation of~(\ref{solcase1}) becomes
\begin{equation}\label{solcase1.1}
(1-2c)x_1^2-x_1x_3+(b+c)x_3^2=0.
\end{equation}
Thus, we have the following singular points
\begin{equation}\label{solcase1.2}
(x_1,x_2,x_3)=\big(2(b+c)q,2(b+c)q,\mu q\big),
\end{equation}
where $\mu=1\pm\sqrt{1{-}4(1{-}2c)(b{+}c)}$, $q\in \mathbb{R}$, $q >0$.
We~observe that, for $c=1/2$, we have only one family of singular points
$(x_1,x_2,x_3)=\big((b+c)q,(b+c)q,q\big)$. Otherwise,
$1-2c>0$, and all depends on the~sign of the~discriminant $D_1=1-4(1-2c)(b+c)$.
Indeed,
there exist one family of singular points for $D_1=0$,
two families for $D_1>0$, and none for $D_1<0$.

If~$x_2\ne x_1$ then $x_3=2b(x_1+x_2)$, so the~second equation of~(\ref{solcase1}) reduces to

\begin{equation}\label{solcase1.3}
(b+c)(1-4b^2)x_1^2-\big(1-2b+8b^2(b+c)
                   \big)x_1x_2+(b+c)\big(1-4b^2
                                    \big)x_2^2=0.
\end{equation}
If~$b=1/2$ then  equation (\ref{solcase1.3}) has no solution. Otherwise, $1-4b^2>0$,
hence, all real roots of the~equation~(\ref{solcase1.3}) are positive.
The~discriminant~$D_2$ of~(\ref{solcase1.3}) has the~same sign as $T:=1-4b-2c+16b^2(b+c)$.
Thus, there exist one family of singular points for~$T=0$,
two families for $T>0$, and none for $T<0$.

In particular, if $a_1=a_2=a_3=a$, $a\in (0,1/2)$, then, for
$a\neq 1/4$ we get exactly four singular points $(x_1,x_2,x_3)$ up to a positive multiple, namely
$(1,1,1)$, $(1-2a,2a,2a)$, $(2a,1-2a,2a)$, or
$(2a,2a,1-2a)$. For~$a=1/4$ we get only singular points proportional to $(1,1,1)$.
\smallskip

{\it The case of pairwise distinct $a_i$'s.} We consider two subcases here.

{\it The case $a_1+a_2+a_3=1/2$.}
Then all singular points~$(x_1,x_2,x_3)$ have the~form
\begin{equation}\label{solcase2}
\aligned
&\big((1-2a_1)q,\,(1-2a_2)q,\,2(a_1+a_2)q
 \big),\\
&\big((1-2a_1)q,\,(1-2a_2)q,\,2(1-a_1-a_2)q
 \big),\\
&\big((1-2a_1)q,\,(1+2a_2)q,\,2(a_1+a_2)q
 \big),\\
&\big((1+2a_1)q,\,(1-2a_2)q,\,2(a_1+a_2)q
 \big),
\endaligned\qquad\text{where}\ \ q\in \mathbb{R},\, q >0 \,.
\end{equation}

{\it The case $a_1+a_2+a_3\ne 1/2$.}
We  look for  singular points of the form $(x_1,x_2,x_3)=(1,t,s)$.
Then~system~(\ref{two_equat_sing1})
 be reduces to
an equation of degree $4$ either in~$s$ or in~$t$. By eliminating the~summand
containing~$t^2$, we obtain
the~following system equivalent to~(\ref{two_equat_sing1}):
\begin{equation}\label{solcase3}
\begin{array}{l}
\bigl( (a_2+a_3)s -(a_1+a_2) \bigr) t=
2a_1(a_2+a_3)s^2+(a_3-a_1)s-2a_3(a_1+a_2),\\
(a_2+a_3)t^2-(a_2+a_3)s^2+s-t+a_3-a_2=0.
\end{array}
\end{equation}

It is easy to see that
$(a_2+a_3)s-(a_1+a_2)\ne 0$
(see details in \cite{Lomshakov2}).
By expressing~$t$ from the~first equation of~(\ref{solcase3})
and inserting it into the~second, we obtain the~following
equation of degree $4$:
\begin{eqnarray}\label{solcase3.1}
(a_2+a_3)^2(2a_1-1)(2a_1+1)s^4+(a_2+a_3)(2a_2+4a_1a_3+1-4a^2_1)s^3 \notag\\
+\Big(2a^2_1{+}2a^2_3{-}8a_1a^2_2a_3{-}2a^2_2{-}
      8a^2_1a_2a_3{-}2a_2{-}
      8a_1a_2a^2_3{-}2a_1a_3{-}a_1{-}a_3{-}8a^2_1a^2_3
 \Big)s^2\\
+(a_1+a_2)(4a_1a_3+
2a_2+1-4a^2_3)s+(2a_3-1)(2a_3+1)(a_1+a_2)^2=0.
\notag
\end{eqnarray}
Denote by~$D_3$ the~discriminant of the~polynomial
in~the left-hand side of~(\ref{solcase3.1}). It can be shown (\cite{Lomshakov2}) that all~real solutions of~(\ref{solcase3.1}) are positive.
For~$D_3\ne 0$, the~equation~(\ref{solcase3.1}) has either two or four distinct real solutions.
Therefore, we get two or four
families of singular points determined by~(\ref{solcase3.1}).
\smallskip

Note, that the~condition~$D_3\,{=}\,0$ holds, for example, for the homogeneous space \linebreak
$SO(20)/\bigl(SO(5)\times SO(6)\times SO(9)\bigr)$ ($a_1=5/36$, $a_2=1/6$, $a_3=1/4$). In this special case,
the~equation~(\ref{solcase3.1}) has one root of multiplicity~2 and the~space under consideration
admits exactly three pairwise nonhomothetic singular points (i.e. invariant Einstein metrics).

\section{The study of singular points}
In this section we recall some facts about the type of singular points of the system~(\ref{two_equat}).
Our basic references are \cite{Dumort} and  \cite{JiangLlible}.
The functions $\widetilde f(x_1,x_2)$ and $\widetilde g(x_1,x_2)$ (see (\ref{two_equat}))
are analytic in a neighborhood of the arbitrary point  $(x_1^0,x_2^0)$ (where $x_1^0>0$ and $x_2^0>0$) and
the following representations are valid:
$$\widetilde f(x_1,x_2) \equiv J_{11}(x_1-x_1^0)+J_{12}(x_2-x_2^0)+F(x_1,x_2),$$
$$\widetilde g(x_1,x_2) \equiv J_{21}(x_1-x_1^0)+J_{22}(x_2-x_2^0)+G(x_1,x_2),$$
where $J_{11}$, $J_{12}$, $J_{21}$ and $J_{22}$ are the elements of the Jacobian matrix

 \begin{equation} \label{linear}
J:=J(x_1^0,x_2^0)=
\left(
\begin{array}{cc}
\dfrac{\partial\widetilde f(x_1^0,x_2^0)}{\partial x_1}&\dfrac{\partial\widetilde f(x_1^0,x_2^0)}{\partial x_2} \\
\dfrac{\partial\widetilde g(x_1^0,x_2^0)}{\partial x_1}&\dfrac{\partial\widetilde g(x_1^0,x_2^0)}{\partial x_2} \\
\end{array}
\right).
\end{equation}
The functions $F$ and $G$ are also analytic in a neighborhood of the point $(x_1^0,x_2^0)$ and
\begin{eqnarray*} \label{FG_0}
F(x_1^0,x_2^0)=G(x_1^0,x_2^0)
=\dfrac  {\partial F(x_1^0,x_2^0)}{\partial x_1} =\dfrac  {\partial F(x_1^0,x_2^0)}{\partial x_2} =
\dfrac  {\partial G(x_1^0,x_2^0)}{\partial x_1} =\dfrac  {\partial G(x_1^0,x_2^0)}{\partial x_2} =0.
\end{eqnarray*}

The eigenvalues of $J(x_1^0,x_2^0)$ can be found from the formula
 \begin{equation*} \label {eigen}
\lambda_{1,2}=\frac{\rho \pm \sqrt {\sigma}}{2}, \quad (|\lambda_1 | \leq |\lambda_2|),
 \end{equation*}
where
\begin{equation}\label{sirhodel}
\sigma:=\rho^2-4\delta, \quad \rho:=\trace\left(J(x_1^0,x_2^0) \right), \quad
\delta:=\det \left( J(x_1^0,x_2^0) \right).
\end{equation}
We will use  these notations in the sequel.
\bigskip

In the non-degenerate case ($\delta=\lambda_1\lambda_2 \ne 0$) we will use the following theorem.
\begin{theorem}[Theorem 2.15 in \cite{Dumort}] \label{nondegenerate}
Let $(0,0)$ be an isolated singular point of the system
$$
\dfrac {dx}{dt} =ax+by + A(x,y), \quad
\dfrac {dy}{dt}=cx+dy+B(x,y),
$$
where $A$ and $B$ are analytic in a neighborhood of the origin with
$$
A(0,0)=B(0,0)=\dfrac  {\partial A(0,0)}{\partial x} =\dfrac  {\partial A(0,0)}{\partial y} =
\dfrac  {\partial B(0,0)}{\partial x} =\dfrac  {\partial B(0,0)}{\partial y} =0.
$$
Let $\lambda_1$ and $\lambda_2$  be the eigenvalues of the matrix
$
\left(
\begin{array}{cc}
a&b\\
c&d
\end{array}
\right)
$
which represents the linear part of the system at the origin. Then the following statements hold:

(i) If $\lambda_1$ and $\lambda_2$ are real and $\lambda_1\lambda_2<0$, then $(0,0)$ is a saddle.

(ii) If  $\lambda_1$ and $\lambda_2$ are real with $|\lambda_1 | \leq |\lambda_2|$ and $\lambda_1\lambda_2>0$,
then $(0,0)$ is a node.

If $\lambda_1>0$ (respectively $<0$) then it is unstable (respectively stable).

(iii) If $\lambda_1=\alpha+ i \beta $ and  $\lambda_2=\alpha- i \beta $ with $\alpha,\beta \ne 0$, then $(0,0)$ is a strong focus.

(iv) If $\lambda_1=i \beta $ and  $\lambda_2=- i \beta $ with $\beta \ne 0$, then $(0,0)$ is a weak focus or a center.
\end{theorem}

Cases (i), (ii) and (iii) are known as {\it hyperbolic singular points}.

\smallskip

A direct calculation of $\delta$ and $\rho$ is often very complicated,
so we will obtain more convenient formulas for $\rho$ and $\delta$ in the case of
{\it singular points $(x_1^0,x_2^0)$ of the system (\ref{two_equat})}.
Let $\widetilde{J}$ be the Jacobian matrix
of the map
$$
(x_1,x_2,x_3) \mapsto \Bigl( f(x_1,x_2,x_3), g(x_1,x_2,x_3), h(x_1,x_2,x_3) \Bigr)
$$
and let $p(t)=t^3-\widetilde{\rho}t^2+\widetilde{\delta}t+\widetilde{\kappa}$
be the characteristic polynomial of $\widetilde{J}$.

\begin{lem} \label{delta_rhon}
If $(x_1^0,x_2^0)$ is a singular point of the system (\ref{two_equat}), then
$\rho=\widetilde{\rho}$ and $\delta=\widetilde{\delta}$, where
$p(t)=t^3-\widetilde{\rho}t^2+\widetilde{\delta}t+\widetilde{\kappa}$ is
calculated at the point $(x_1^0,x_2^0, x_3^0=\varphi(x_1^0,x_2^0))$.
\end{lem}

\begin{proof}
Since the volume $V$ is a first integral of the system (\ref{three_equat}), then
$\frac{\partial{V}}{\partial x_1}f+\frac{\partial{V}}{\partial x_2}g+\frac{\partial{V}}{\partial x_3}h \equiv 0$.
On the other hand, $\varphi _{x_1}=-\frac{\partial{V}}{\partial x_1}/\frac{\partial{V}}{\partial x_3}$ and
$\varphi _{x_2}=-\frac{\partial{V}}{\partial x_2}/\frac{\partial{V}}{\partial x_3}$, hence
$h=f\varphi _{x_1}+g\varphi _{x_2}$. Therefore, for any $i=1,2,3$ we have
$
\frac{\partial h}{\partial x_i}=\frac{\partial f}{\partial x_i}\varphi _{x_1}+
\frac{\partial g}{\partial x_i}\varphi _{x_2} +
f\varphi _{x_1x_i}+g\varphi _{x_2x_i}.
$
At a singular point ($\widetilde{f}=f=\widetilde{g}=g=0$) we have
\begin{equation*}\label{simpltdif}
\frac{\partial h}{\partial x_i}=\frac{\partial f}{\partial x_i}\varphi _{x_1}+
\frac{\partial g}{\partial x_i}\varphi _{x_2}.
\end{equation*}
Hence,
$$
\widetilde{\rho}=\frac{\partial f}{\partial x_1}+\frac{\partial g}{\partial x_2}+\frac{\partial h}{\partial x_3}=
\frac{\partial f}{\partial x_1}+\frac{\partial g}{\partial x_2}+\frac{\partial f}{\partial x_3}\varphi _{x_1}+
\frac{\partial g}{\partial x_3}\varphi _{x_2}=\rho
$$
at any singular point. By the same manner and by using (\ref{simpltdif}) we get

\begin{eqnarray*}
\widetilde{\delta}&=&
\frac{\partial f}{\partial x_1}\frac{\partial g}{\partial x_2}-
\frac{\partial f}{\partial x_2}\frac{\partial g}{\partial x_1}+
\frac{\partial f}{\partial x_1}\frac{\partial h}{\partial x_3}-
\frac{\partial f}{\partial x_3}\frac{\partial h}{\partial x_1}+
\frac{\partial g}{\partial x_2}\frac{\partial h}{\partial x_3}-
\frac{\partial g}{\partial x_3}\frac{\partial h}{\partial x_2}\\
&=&\frac{\partial f}{\partial x_1}\frac{\partial g}{\partial x_2}-
\frac{\partial f}{\partial x_2}\frac{\partial g}{\partial x_1}+
\varphi _{x_1}
\left(
\frac{\partial g}{\partial x_2}\frac{\partial f}{\partial x_3}-
\frac{\partial g}{\partial x_3}\frac{\partial f}{\partial x_2}
\right)+
\varphi _{x_2}
\left(
\frac{\partial f}{\partial x_1}\frac{\partial g}{\partial x_3}-
\frac{\partial g}{\partial x_1}\frac{\partial f}{\partial x_3}
\right)=\delta
\end{eqnarray*}
at any singular point.
\end{proof}

\begin{center}
\begin{figure}[t]
\centering\scalebox{1}[1]{\includegraphics[angle=-90,totalheight=4in]
{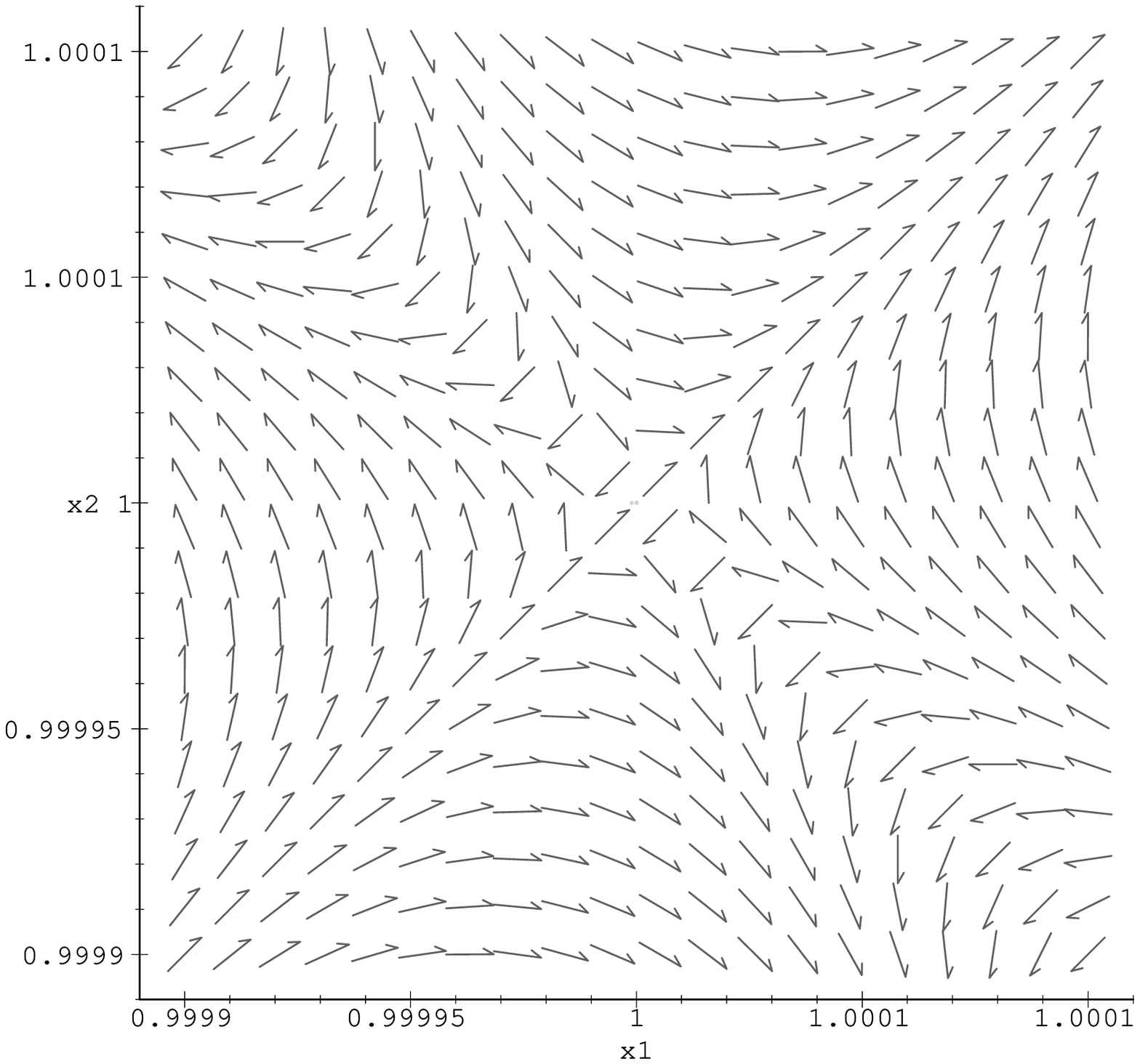}}
\caption{}
\label{degsing}
\end{figure}
\end{center}

\section{A special case}\label{onespc}

Here we consider the case $a_1=a_2=a_3=1/4$ which has a special interest.
\begin{theorem} \label{blow_up}
At $a_1=a_2=a_3=1/4$ the system of ODE's (\ref{two_equat}) has an unique (isolated) singular point
$(x_1^0,x_2^0)=(1,1)$ which is a saddle with six hyperbolic sectors.
\end{theorem}
\smallskip
\begin{proof}
As the calculations show the unique singular point of (\ref{two_equat}) is indeed $(x_1^0, x_2^0)=(1,1)$
(see Figure \ref{degsing} for a phase portrait in a neighborhood of this point).
Note that in this case $J(x_1^0,x_2^0)=\left (
\begin{array}{cc}
0&0 \\
0&0 \\
\end{array}
\right).
$
By  moving $(1,1)$ to the origin and by the analyticity of $\tilde f$ and $\tilde g$ at $(1,1)$,   system
(\ref{two_equat}) can be reduced to the equivalent system
\begin{eqnarray*}
\dfrac{dx}{dt}&=&P_2(x,y)+P_3(x,y)+P_4(x,y)+\ldots, \\
\dfrac{dy}{dt}&=&Q_2(x,y)+Q_3(x,y)+Q_4(x,y)+\ldots,
\end{eqnarray*}
where
$$P_2(x,y)=-x^2/2+xy+y^2,  \quad Q_2(x,y)=x^2+xy-y^2/2,$$
and $P_i(x,y), Q_i(x,y)$ ($i \geq 3$)  are some homogeneous polynomials of degree $i$ with respect to $x$ and $y$
($x=x_1-1$, $y=x_2-1$).

Next we use  results from \cite{JiangLlible}. Using the blowing-up $y=ux$, $d\tau=xdt$ we obtain the system
\begin{eqnarray*}
\dfrac{dx}{d\tau}&=&xP_2(1,u)+x^2P_3(1,u)+x^3P_4(1,u)+\ldots,\\
\dfrac{du}{d\tau}&=&\Delta(u)+x \big(Q_3(1,u)-uP_3(1,u) \big)+x^2\big(Q_4(1,u)-uP_4(1,u)\big)+\ldots,
\end{eqnarray*}
where
$$
P_2(1,u)=u^2+u-1/2, \quad \Delta(u)=-(u-u_1) (u-u_2)(u-u_3),
$$
$$
u_1=-2, \quad u_2=-1/2, \quad u_3=1.
$$
We have the case of Subsection 6.2 of \cite{JiangLlible} where the equation $\Delta(u)=0$ has three different real roots.
So it is obvious that the  blowing-up system has the singular points $(0,u_1)$, $(0,u_2)$ and $(0,u_3)$.
We show that all of these singular points are saddles.

Let $\beta_i:=P_2(1,u_i)$ and let $J_i$ be the matrix of the linear part of the blowing-up system at the point
$(0,u_i)$, $i=1,2,3$.
Then
$$
J_i=
\left(
\begin{array}{cc}
\beta_i&0 \\
\dfrac{(u_i-1)(u_i+1)(2u_i^2-u_i+2)}{2}&-3\beta_i \\
\end{array}
\right)
$$
with eigenvalues equal to $\beta_i$ and $-3\beta_i$.
It is clear that $\beta_i \ne 0$ for all $i=1,2,3$. Therefore the eigenvalues of $J_i$ have different signs,
and by Theorem  \ref{nondegenerate} all the singular points $(0,u_i)$, $i=1,2,3$, are saddles. The phase
portrait of the blowing-up system is identical to one shown in Figure~7(b) \cite{JiangLlible}.

The saddles $(0,u_i)$, $i=1,2,3$, correspond to the unique singular point $(0,0)$ of the initial system.
According to the qualitative classification of singular points of degree $2$ given by \cite{JiangLlible} the
point $(0,0)$ is also a saddle with six hyperbolic sectors near it (see Figure~3(12) in \cite{JiangLlible}).
\end{proof}

\section{The ``degeneration'' set $\Omega$}\label{degenerset}

Recall that the systems (\ref{three_equat}) and (\ref{two_equat}) are well defined for all
$(a_1,a_2,a_3)\in\mathbb{R}^3$ with $a_1a_2+a_1a_3+a_2a_3\neq 0$
(with the additional restriction $a_1a_2a_3\neq 0$ for the system (\ref{two_equat})).

The special case considered in Section \ref{onespc} leads us to consider the set
$$
\Omega =\{(a_1,a_2,a_3)\in\mathbb{R}^3\,|\,  \mbox{system (\ref{two_equat}) has at least
one degenerate singular point}\}.
$$

We will show a convenient way to deal with degenerate singular points of the system (\ref{two_equat})
in Lemma \ref{sing.one} below.

A direct calculation using Lemma \ref{delta_rhon} gives the following:

\begin{lem} \label{sing.zero}
Let $p(t)=t^3-\widetilde{\rho}t^2+\widetilde{\delta}t+\widetilde{\kappa}$ be the characteristic polynomial
of the Jacobi matrix of the system (\ref{three_equat}) at
any point $(x_1,x_2,x_3)$ with $x_1x_2x_3\neq 0$.
Then  $\widetilde{\kappa}=0$,
$\widetilde{\rho}=\frac{2F_1}{\mathcal{A}x_1x_2x_3}$ and $\widetilde{\delta}=\frac{F_2}{\mathcal{A}^2x_1^2x_2^2x_3^2}$,
where
\begin{eqnarray}\label{singval0}\notag
F_1\,\,\,\,=\,\,\,\,a_1a_2x_1x_2+a_1a_3x_1x_3+a_2a_3x_2x_3\\
-(\mathcal{A}+a_2a_3)a_1x_1^2-(\mathcal{A}+a_1a_3)a_2x_2^2-(\mathcal{A}+a_1a_2)a_3x_3^2,
\end{eqnarray}
\begin{eqnarray}\label{singval1}\notag
F_2\,\,\,\,=\,\,\,\Bigl(a_1^2a_2a_3(2\mathcal{A}+a_2a_3)-\mathcal{A}^3\Bigr)x_1^4+
\Bigl(a_1a_2^2a_3(2\mathcal{A}+a_1a_3)-\mathcal{A}^3\Bigr)x_2^4\\ \notag
+\Bigl(a_1a_2a_3^2(2\mathcal{A}+a_1a_2)-\mathcal{A}^3\Bigr)x_3^4-2a_1^2(\mathcal{A}+a_2a_3)(a_2x_2+a_3x_3)x_1^3\\ \notag
-2a_2^2(\mathcal{A}+a_1a_3)(a_1x_1+a_3x_3)x_2^3-2a_3^2(\mathcal{A}+a_1a_2)(a_2x_2+a_1x_1)x_3^3\\
+\Bigl(a_1a_2(2(3a_1a_2+a_2^2+a_1^2)a_3^2+2(a_1+a_2)a_1a_2a_3+a_1a_2)+2\mathcal{A}^3\Bigr)x_1^2x_2^2\\ \notag
+\Bigl(a_1a_3(2(3a_1a_3+a_3^2+a_1^2)a_2^2+2(a_1+a_3)a_1a_2a_3+a_1a_3)+2\mathcal{A}^3\Bigr)x_1^2x_3^2\\ \notag
+\Bigl(a_2a_3(2(3a_2a_3+a_3^2+a_2^2)a_1^2+2(a_2+a_3)a_1a_2a_3+a_2a_3)+2\mathcal{A}^3\Bigr)x_2^2x_3^2\\ \notag
-2a_1a_2a_3\Bigl((\mathcal{A}+a_2a_3-a_1)x_1+(\mathcal{A}+a_1a_3-a_2)x_2+
(\mathcal{A}+a_1a_2-a_3)x_3\Bigr)x_1x_2x_3, \notag
\end{eqnarray}
and $\mathcal{A}=a_1a_2+a_1a_3+a_2a_3$.
\end{lem}

\begin{lem} \label{sing.one}
A singular point $(x_1,x_2)$ of the system (\ref{two_equat}) is degenerate if and only
if the point $(x_1,x_2, x_3=\varphi(x_1,x_2))$ satisfies the equation $F_2=0$, where
$F_2$ is given by (\ref{singval1}).
\end{lem}

\begin{proof} By Lemma \ref{delta_rhon}  a singular point $(x_1,x_2)$  of the system (\ref{two_equat}) is degenerate
if and only if $\delta=\widetilde{\delta}=0$.
Note that the right hand side of  equation (\ref{singval1}) is homogeneous in the variables $x_1,x_2,x_3$.
Therefore, without loss of generality we may consider these variables up to a positive multiple.
Obviously from  Lemma \ref{sing.zero} it follows that $\widetilde{\delta}=0$ is equivalent to $F_2=0$.
\end{proof}
\bigskip

Now, we can represent the set $\Omega$ as an algebraic surface in $\mathbb{R}^3$.

\begin{lem} \label{sing.two}
A point $(a_1,a_2,a_3)$ with $a_1a_2+a_1a_3+a_2a_3\neq 0$ and $a_1a_2a_3\neq 0$
lies in the set $\Omega$ if and only if
$Q(a_1,a_2,a_3)=0$, where
\begin{eqnarray}\label{singval2}\notag
Q(a_1,a_2,a_3)\,=\,
(2s_1+4s_3-1)(64s_1^5-64s_1^4+8s_1^3+12s_1^2-6s_1+1\\\notag
+240s_3s_1^2-240s_3s_1-1536s_3^2s_1-4096s_3^3+60s_3+768s_3^2)\\
-8s_1(2s_1+4s_3-1)(2s_1-32s_3-1)(10s_1+32s_3-5)s_2\\\notag
-16s_1^2(13-52s_1+640s_3s_1+1024s_3^2-320s_3+52s_1^2)s_2^2\\\notag
+64(2s_1-1)(2s_1-32s_3-1)s_2^3+2048s_1(2s_1-1)s_2^4,
\end{eqnarray}
and
$$
s_1 = a_1+a_2+a_3, \quad s_2 = a_1a_2+a_1a_3+a_2a_3, \quad s_3 = a_1a_2a_3.
$$
\end{lem}

\begin{proof}
Under the assumptions $a_1a_2+a_1a_3+a_2a_3\neq 0$, $a_1a_2a_3\neq 0$, and $x_3=\varphi(x_1,x_2)$,
the singular points $(x_1,x_2)$ of the
system (\ref{two_equat}) can be found from the equations (\ref{two_equat_sing1}),
and according to Lemma~\ref{sing.one} they are degenerate if and only if $F_2=0$.
Note that the equations (\ref{two_equat_sing1}) and $F_2=0$
are homogeneous with respect to $x_1,x_2,x_3$. Setting $x_3=1$ and eliminating $x_1$ and $x_2$
from these three equations (e.~g. using {\sl Maple} or {\sl Mathemathica}) we obtain the equation
$$
(4a_1^2-1)(4a_2^2-1)(a_1+a_3)(a_2+a_3)(a_1a_2+a_1a_3+a_2a_3)^2\cdot Q(a_1,a_2,a_3)=0.
$$
Note that for $a_1=\pm 1/2$, $a_2=\pm 1/2$, $a_3=-a_1$, and $a_3=-a_2$ we have no additional
triples  $(a_1,a_2,a_3)$ of ``degenerate'' parameters.
Therefore, the first four factors in the above equation can be ignored.
Another reason to ignore them is the symmetry of the problem under the permutation $a_1\to a_2 \to a_3\to a_1$.
\end{proof}

It is easy to see that $Q(a_1,a_2,a_3)$ is a symmetric polynomial in $a_1,a_2,a_3$
of degree 12. Therefore, the equation $Q(a_1,a_2,a_3)=0$
(without the restrictions $a_1a_2+a_1a_3+a_2a_3\neq 0$ and $a_1a_2a_3\neq 0$)
defines an algebraic surface in
$\mathbb{R}^3$ that we
may identify with the set $\Omega$.

\begin{center}
\begin{figure}[t]
\centering\scalebox{1}[1]{\includegraphics[angle=-90,totalheight=3in]
{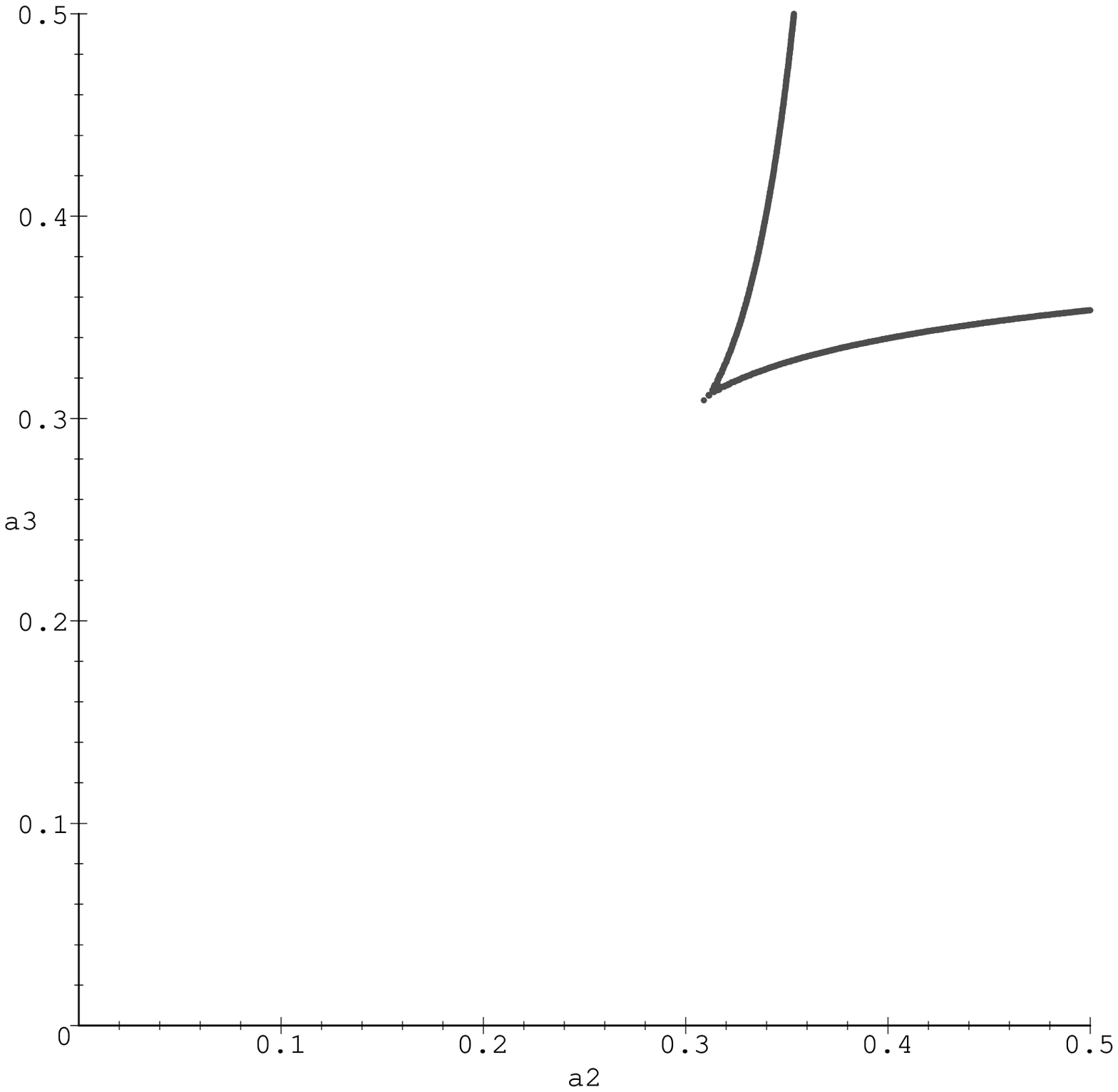}}
\caption{}
\label{pictur1}
\end{figure}
\end{center}

{\it In the rest of this section, we consider only points
$(a_1,a_2,a_3)\in(0,1/2]\times(0,1/2]\times (0,1/2]$.}

It is very important to describe in details the set
$$
\Omega \cap (0,1/2]^3 =\{(a_1,a_2,a_3)\in (0,1/2]\times (0,1/2]\times (0,1/2]\,: \,Q(a_1,a_2,a_3)=0\}.
$$
As usual, the most complicated and interesting problem is
the study of this surface in neighborhoods of singular points of $\Omega$
determined by $\nabla Q(a_1,a_2,a_3)=0$.
\medskip

For $a_1=1/2$ the equation $Q=0$ is equivalent to
$$
4\widetilde{s}_2(4\widetilde{s}_2+1)^2-4(4\widetilde{s}_2-1)(4\widetilde{s}_2+1)^2\widetilde{s}_1-
13(4\widetilde{s}_2+1)^2\widetilde{s}_1^2+4(4\widetilde{s}_2-1)\widetilde{s}_1^3+44\widetilde{s}_1^4=0,
$$
where $\widetilde{s}_1=a_2+a_3$ and $\widetilde{s}_2=a_2a_3$. If $a_2,a_3 \in (0,1/2]$ then this set is
a curve homeomorphic to the interval $[0,1]$ with endpoints $(1/2,1/2,\sqrt{2}/2)$ and $(1/2,\sqrt{2}/4,1/2)$
and with the singular point (a cusp) at the point $a_3=a_2=(\sqrt{5}-1)/4\approx 0.3090169942$
(see Figure \ref{pictur1}).
The same is also valid under the permutation $a_1\to a_2 \to a_3\to a_1$.

Note that for $s_1=a_1+a_2+a_3=1/2$ the equation $Q=0$ is equivalent to
$s_3^2(s_2-2s_3)^2=\frac{1}{2}a_1a_2a_3(-5a_1a_2+a_1-2a_1^2+a_2-2a_2^2+6a_1a_2(a_1+a_2))=0$.
It is easy to check that for $a_1,a_2 \in [0,1/2]$  the equality $-5a_1a_2+a_1-2a_1^2+a_2-2a_2^2+6a_1a_2(a_1+a_2)=0$
holds only when $(a_1,a_2)$ is one of the  points $(0,0)$, $(0,1/2)$, and $(1/2,0)$.

Therefore, $s_1=a_1+a_2+a_3=1/2$
in the set $\Omega \cap [0,1/2]^3$ only for points in the boundary of the triangle with  vertices
$(0,0,1/2)$, $(0,1/2,0)$, and $(1/2,0,0)$. For all other points in $\Omega \cap (0,1/2]^3$ we have the inequality
$s_1=a_1+a_2+a_3>1/2$.
\smallskip

It is clear that $(1/4,1/4,1/4)\in \Omega$.
Note that for $s_1=a_1+a_2+a_3=3/4$ the equation $Q=0$ is equivalent to
$$
(24s_2^2+8s_2+64s_2s_3-8s_3-128s_3^2+1)(32s_2-64s_3-5)^2=0.
$$
It is not difficult to show that $(1/4,1/4,1/4)$ is the only point in $\Omega \cap [0,1/2]^3$
satisfying the additional condition $s_1=a_1+a_2+a_3=3/4$.

It turns out that the point $(1/4,1/4,1/4)$ is a singular point of degree $3$ of the algebraic surface
$\Omega$ (see Figure \ref{singsur}). The type of this point is {\bf elliptic umbilic} in the sense of Darboux
(see \cite[pp.~448--464]{Dar} and \cite[p.~320]{Thom}) or of type
$D_4^-$ in other terminology (see e.~g. \cite[Chapter III, Sections 21.3, 22.3]{AGV}).
\smallskip

\section{On the signs of $\sigma$ and $\rho$ for singular points of the system (\ref{two_equat})}

In this section we study  the singular
points of the system~(\ref{two_equat}) according to the signs of $\sigma$ and $\rho$ (see (\ref{sirhodel})). Results depend on conditions on the parameters $a_1, a_2, a_3$.

\begin{lem} \label{sigma_minim}
The quadratic form
\begin{eqnarray*}
G(x,y,z):=-2(a_1a_3+a_2a_3)xy-2(a_1a_2+a_1a_3)yz-2(a_2a_3+a_1a_2)xz\\
+(\mathcal{A}+a_1^2)x^2+(\mathcal{A}+a_2^2)y^2+(\mathcal{A}+a_3^2)z^2,
\end{eqnarray*}
where $\mathcal{A}=a_1a_2+a_1a_3+a_2a_3$,
is non-negative if $\mathcal{A}>0$ (in particular, if $a_i>0$ for $i=1,2,3$)
and achieves its absolute minimum (equal to zero) exactly at the points
$$
(x,y,z)=\bigl((a_2+a_3)t, (a_1+a_3)t,(a_1+a_2)t \bigr), \quad t \in \mathbb{R}.
$$
\end{lem}

\begin{proof}
It is easy to show that the matrix of the form $G$ has non negative eigenvalues
$0$, $2\mathcal{A}$, and $a_1^2+a_2^2+a_3^2+\mathcal{A}$. Obviously, the last two numbers are positive for $\mathcal{A}>0$.
Therefore $G$ is a non-negative form. Note that the equation $G=0$
has the solutions given in the statement of the lemma.
\end{proof}

\begin{theorem}\label{nonode}
For $a_1a_2+a_1a_3+a_2a_3>0$ all singular points of the system (\ref{two_equat}) are such that
$\sigma\equiv\rho^2-4\delta\geq 0$. In particular, a non-degenerate singular point (i.e. $\delta\neq 0$) of  (\ref{two_equat}) is either a node (if $\delta >0$) or a saddle (if $\delta <0$).
\end{theorem}

\begin{proof}
Let $\mathcal{A}=a_1a_2+a_1a_3+a_2a_3>0$. By Lemmas \ref{delta_rhon} and \ref{sing.zero} we get
\begin{eqnarray*}
\sigma={\rho}^2-4{\delta}=\widetilde{\rho}^2-4\widetilde{\delta}=\\
\Bigl(-2(a_1a_3+a_2a_3)x_1^2x_2^2-2(a_1a_2+a_1a_3)x_2^2x_3^2-2(a_2a_3+a_1a_2)x_1^2x_3^2+\\
(\mathcal{A}+a_1^2)x_1^4+(\mathcal{A}+a_2^2)x_2^4+(\mathcal{A}+a_3^2)x_3^4\Bigr)
x_1^{-2}x_2^{-2}x_3^{-2}.
\end{eqnarray*}
Using Lemma \ref{sigma_minim} for $x=x_1^2$, $y=x_2^2$ and $z=x_3^2$, we get that $\sigma \geq 0$
for all $x_1,x_2,x_3>0$, $x_3= \varphi(x_1,x_2)$.
So, in particular, $\sigma \geq 0$  holds for all singular points of the system (\ref{two_equat}), and
by Theorem \ref{nondegenerate} its non degenerate singular points can be only
either a node or a saddle.
\end{proof}

\begin {remark}\label{sigmaziro}
From Theorem \ref{nonode} and Lemma \ref{sigma_minim} we get the following. If  $(x_1,x_2)$ is a singular point
of the system (\ref{two_equat})  with $\sigma=0$ then
\begin{equation}
\label{sigma=0}
(x_1,x_2,x_3)=\Bigl(q\sqrt {a_2+a_3},q\sqrt{a_1+a_3},q\sqrt{a_1+a_2} \Bigr)
\end{equation}
for a unique $q \in \mathbb{R}$, $q>0$, determined by  the equality $x_3= \varphi(x_1,x_2)$.
\end {remark}
Next we are interested for those values $(a_1,a_2,a_3)\in(0,1/2]\times(0,1/2]\times (0,1/2]$ such that
the system (\ref{two_equat}) has at least one singular point with $\sigma=0$.

\begin{theorem} \label{a_i_ensur_sing_points_sigma=0}
The only two families of the parameters $a_i$, $i=1,2,3$ satisfying the conditions $a_i\in (0,1/2]$
and which can give singular points of the system (\ref{two_equat}) with the property $\sigma=0$, are
the following:
\begin{equation} \label{first_for_sigma=0}
a_1=a_2=a_3=s, \quad s\in (0,1/2],
\end{equation}
\begin{equation} \label{second_for_sigma=0}
a_i=a_j= \frac{(2s^2-1)^2}{8s^2}, \quad a_k= \frac{4s^4+4s^2-1}{8s^2}, \quad s \in (s_1,s_2),
\end{equation}
where  $s_1:= \sqrt{2\sqrt{2}-2}/2$, $s_2:=\sqrt{2}/2$\ \   ($i,j,k \in \{1,2,3\}$,  $i\neq j\neq k\neq i$).
\end{theorem}

\begin{proof}
Consider the system of equations
(\ref {two_equat_sing1}) and (\ref {sigma=0}) under the substitutions
$$
b_1:=\sqrt{a_2+a_3}, \quad b_2:=\sqrt{a_1+a_3}, \quad b_3:=\sqrt{a_1+a_2}.
$$
Then
$$
a_i= \Bigl(b_j^2+b_k^2-b_i^2\Bigr)/2, \quad x_i= qb_i,
$$
and  equations (\ref {two_equat_sing1})  take the form
\begin{equation} \label{two_equat_sing1_new}
\begin{array}{l}
2b_1b_3^4+2b_1b_2^4-2b_1^3b_3^2-2b_1^3b_2^2+b_3b_1^2
+b_2b_1^2+b_3b_2^2+b_2b_3^2
-b_2^3-b_3^3-2b_1b_2b_3=0,\\
2b_2b_3^4+2b_2b_1^4-2b_2^3b_3^2-2b_2^3b_1^2+b_3b_2^2+b_1b_2^2
+b_3b_1^2+b_1b_3^2-b_1^3-b_3^3-2b_1b_2b_3=0.\\
\end{array}
\end{equation}
By eliminating $b_1$ (and, independently, $b_2$) from (\ref {two_equat_sing1_new}) we obtain the following two equations:
\begin{equation*} \label{two_equat_sing1_new_sled}
(b_3-b_2)(2b_2^2+2b_2b_3-1)(2b_3^2+2b_2b_3-1)(2b_2^2+2b_3^2+2b_2b_3-1)
(2b_2^2+2b_3^2-2b_2b_3-1)=0,
\end{equation*}
\begin{equation*} \label{two_equat_sing2_new_sled}
(b_3-b_1)(2b_1^2+2b_1b_3-1)(2b_3^2+2b_1b_3-1)(2b_1^2+2b_3^2+2b_1b_3-1)
(2b_1^2+2b_3^2-2b_1b_3-1)=0.
\end{equation*}
Solving this equations together (and
taking into account the equations (\ref {two_equat_sing1_new}))
we get the following four families of solutions of
(\ref {two_equat_sing1_new}), that include solutions under the conditions $b_i>0$ ($i=1,2,3$):
$$
b_1=b_2=b_3=s,
$$
$$
b_i=b_j=s, \quad b_k=\bigl(1-2s^2 \bigl)/(2s),
$$
$$
b_i=\Bigl(s+\sqrt{2-3s^2}\Bigr)/2, \quad b_j=\Bigl(-s+\sqrt{2-3s^2}\Bigr)/2,\quad b_k=s,
$$
$$
b_i=\Bigl(s+\sqrt{2-3s^2}\Bigr)/2, \quad b_j=\Bigl(s-\sqrt{2-3s^2}\Bigr)/2,\quad b_k=s,
$$
where $s \in \mathbb{R}$, $s>0$, \quad $i,j,k \in \{1,2,3\}$, $i\neq j\neq k\neq i$.

Now we return to the original parameters $a_i$ ($i=1,2,3$). From the first family $b_1=b_2=b_3=s$
we obtain (\ref{first_for_sigma=0}).

The second family $b_i=b_j=s$, $b_k=\bigl(1-2s^2 \bigl)/(2s)$ gives (\ref{second_for_sigma=0}).
It is clear that $a_i \in (0,1/2]$ for all $i=1,2,3$ if $s\in (s_1,s_2)$.

The third and fourth families of $b_i$ give
$$
a_i=\frac{s^2-s\sqrt{2-3s^2}}{2}, \quad  a_j=\frac{s^2+s\sqrt{2-3s^2}}{2},   \quad  a_k=-s^2+1/2.
$$
It is easy to show that the system of inequalities $a_i>0$, $i=1,2,3$, is incompatible in this case.
\end{proof}

\begin{remark}\label{sing_points_with_sigma=0}
Now we can find all singular points of the system (\ref{two_equat}) corresponding to the families
(\ref{first_for_sigma=0}), (\ref{second_for_sigma=0}) and having $\sigma=0$.

According to Remark \ref{sigmaziro}
the family (\ref{first_for_sigma=0})  gives a unique  singular point
$(x_1,x_2)=(1,1)$ of~(\ref{two_equat}) satisfying  $\sigma=0$ for all $s\in (0,1/2]$.

Analogously, by Remark \ref{sigmaziro} it follows that the family
(\ref{second_for_sigma=0})  gives only the following singular points of (\ref{two_equat})
satisfying  $\sigma=0$ for all $s \in (s_1,s_2)$:
$$
(2s^2q,2s^2q),\quad \bigl(2s^2q,(1-2s^2)q\bigr), \quad \bigl((1-2s^2)q, 2s^2q\bigr),
$$
where  $q=(2s^2)^{\tfrac{-2(4s^4+4s^2-1)}{(6s^2-1)(2s^2+1)}}(1-2s^2)^{\tfrac{-(2s^2-1)^2}{(6s^2-1)(2s^2+1)}}>0$
is determined by the condition $V\equiv 1$ (see (\ref{volume})).
\end {remark}

\begin {remark} \label{roots_of_Q}
Using Theorem \ref{a_i_ensur_sing_points_sigma=0} and Lemma \ref{sing.two} we can detect all values of
$a_i\in (0,1/2]$, $i=1,2,3$, such that the system (\ref{two_equat}) has at least one singular point with
$\sigma=\delta=0$.
According to Remark \ref{sing_points_with_sigma=0} for all $s\in (0,1/2]$ the family (\ref{first_for_sigma=0})
gives a unique singular point $(x_1,x_2)=(1,1)$ of (\ref{two_equat}) with $\sigma=0$.
In this case $Q$ (see (\ref{singval2})) takes the form
$$
Q=-(2s+1)^4(4s-1)^8,
$$
and the equation $Q=0$ implies that $s=1/4$. Then according to Lemma \ref{sing.two}  the point $(1,1)$ is a degenerate singular
point ($\delta=0$) only for $s=1/4$ (and its type has been determined in Theorem \ref{blow_up}).
For $s\in (0,1/4)\cup(1/4,1/2]$ the point $(1,1)$ is a node.

\smallskip
Analogously, for the family (\ref{second_for_sigma=0}) we have that
$$
Q=s^8(1-8s^2-4s^4)(1-2s^2)^3(3-2s^2)^3,
$$
and the equation $Q=0$ has only three positive roots $\sqrt{2\sqrt{5}-4}/2$, $\sqrt{2}/2$ and
$\sqrt{6}/2$, but none of these values belongs to the interval $(s_1,s_2)$. Therefore,
(\ref{second_for_sigma=0}) can not give singular points of (\ref{two_equat}) with  $\sigma=\delta=0$.
\end {remark}

Next we denote by $S$  the set of points $(a_1,a_2,a_3)$ such that  there is a singular point $(x_1^0,x_2^0)$
of the system
(\ref{two_equat}) with $\rho=0$. Recall  that for points with $a_1a_2+a_1a_3+a_2a_3=0$ the system (\ref{two_equat})
is undefined.

\begin{theorem} \label{sign_rho}
A point $(a_1,a_2,a_3)$  with $a_1a_2+a_1a_3+a_2a_3\neq 0$ and $a_1a_2a_3\neq 0$ lies on the surface $S$ if and only if
\begin{equation}\label{signrho}
Q_1(a_1,a_2,a_3):=4(a_1+a_2)(a_1+a_3)(a_2+a_3)-2a_1-2a_2-2a_3+1=0.
\end{equation}
\end{theorem}
\begin{proof}
The equations (\ref{two_equat_sing1}) and (\ref{singval0})
are homogeneous with respect to $x_1,x_2,x_3$. Now setting $x_3=1$ and eliminating $x_1$ and $x_2$
from the above three equations (using e.~g. {\sl Maple} or {\sl Mathemathica}) we get the equations
$$
(a_1+a_3)(a_2+a_3)(a_1a_2+a_1a_3+a_2a_3)\cdot \bigl(4(a_1+a_2)(a_1+a_3)(a_2+a_3)-2(a_1+a_2+a_3)+1\bigr)=0.
$$
Note that for  $a_3=-a_1$ and $a_3=-a_2$ we have no additional sets of
parameters $(a_1,a_2,a_3)$. This finishes the proof.
\end{proof}
\smallskip

\begin{remark}\label{singdegenp}
It is easy to show that  system (\ref{two_equat}) has a singular point with $\rho=\delta=0$ if and only if
$(a_1,a_2,a_3)=(1/4,1/4,1/4)$ (in this case  system (\ref{two_equat}) has exactly one
degenerate singular point $(x_1,x_2)=(1,1)$, see Section \ref{onespc}).
Indeed, from the equations $\rho=0$ and $\delta=0$ we have $\sigma=\rho^2-4\delta=0$.
According to Remark \ref{roots_of_Q}, the system of  equations $\delta=0$, $\sigma=0$ has a unique solution
$(a_1,a_2,a_3)=(1/4,1/4,1/4)$. It is easy to check that this solution satisfies the equation (\ref{signrho}) as well.
Therefore,  Theorem~\ref{sign_rho} implies that $\rho=0$.
\end{remark}
\smallskip

\section{Singular points for parameters in the set $(0,1/2)^3\setminus \Omega$}

We first discuss a part of the surface $\Omega$ (see Section \ref{degenerset}) in the cube $(0,1/2)^3$.
Recall that $\Omega$ is invariant under the permutation $a_1\to a_2\to a_3\to a_1$. It should be noted that the set
$(0,1/2)^3\cap \Omega$ is connected (it can be shown by lengthy computations
using suitable geometric tools).
There are three curves (``edges'') of {\it singular points} on $\Omega$ (i.~e. points where $\nabla Q =0$):
one of them has parametric representation $a_1=-\frac{1}{2}\frac{16t^3-4t+1}{8t^2-1}, a_2=a_3=t$, and the others
are defined by permutations of $a_i$. These curves  have a common point $(1/4,1/4,1/4)$ which is an elliptic umbilic
on the surface $\Omega$ (see Figure~\ref{singsur}). The part of $\Omega$ in $(0,1/2)^3$
consists of three (pairwise isometric) ``bubbles'' spanned on every pair of ``edges''
(cf. pictures of elliptic umbilics at pp.~64--91 of \cite{WP}).
The Gaussian curvature at every non-singular point of the surface $\Omega\cap (0,1/2)^3$ is negative,
as it could be checked by direct calculations.

\begin{remark}
It should be also noted that the point $(a_1,a_2,a_3)=(1/4,1/4,1/4)$ is the unique singular point of the surface $S$
($\nabla Q_1(1/4,1/4,1/4)=0$).
This is clear by reducing (\ref{signrho}) to the simpler equation
$4z_1z_2z_3-z_1-z_2-z_3+1=0$ using the substitutions $z_1=a_1+a_2, z_2=a_1+a_3, z_3=a_2+a_3$.
It is easy to see that $S$  divides
the cube $[0,1/2]^3$ into three domains $\widetilde{O}_1$, $\widetilde{O}_2$, and $\widetilde{O}_3$ containing the points
$(0,0,0)$, $(1/2,1/2,1/2)$, and $(1/8,1/4,3/8)$ respectively.
\end{remark}

\smallskip
From the above discussion and some geometric considerations (that could be rigorous but very lengthy),
we see that the set $(0,1/2)^3\setminus \Omega$ has exactly three connected components.
Denote by $O_1$, $O_2$, and $O_3$ the components containing the points
$(1/6,1/6,1/6)$, $(7/15,7/15,7/15)$, and $(1/6, 1/4, 1/3)$ respectively.
\smallskip

Let us fix $i \in \{ 1,2,3\}$.
By the definition of $\Omega$, for all points $(a_1,a_2,a_3) \in O_i$  system (\ref{two_equat}) has only
non degenerate singular points.
The number of these points and their corresponding types are the same on each component $O_i$
(under some suitable identification for various values of parameters $a_1,a_2,a_3$).
Therefore, it suffices to check {\it only one point in the  set $O_i$}.

One of the main results of this paper is the following theorem which
clarifies the above observation and provides a general result about the type of the non degenerate
singular points of the system
(\ref{two_equat}).

\begin{center}
\begin{figure}[t]
\centering\scalebox{1}[1]{
\includegraphics[angle=-90,totalheight=2.5in]{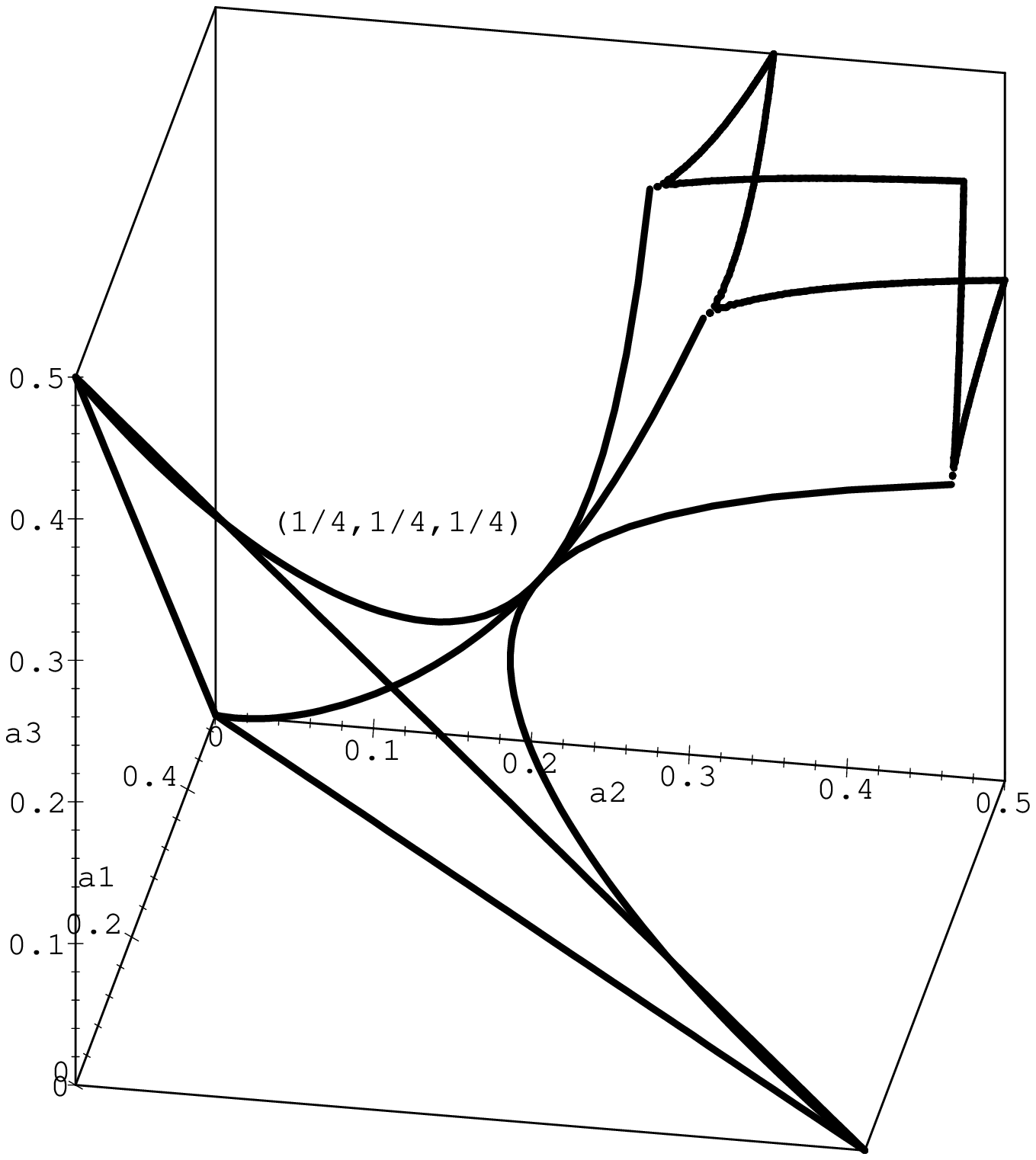}
\includegraphics[angle=-90,totalheight=2.5in]{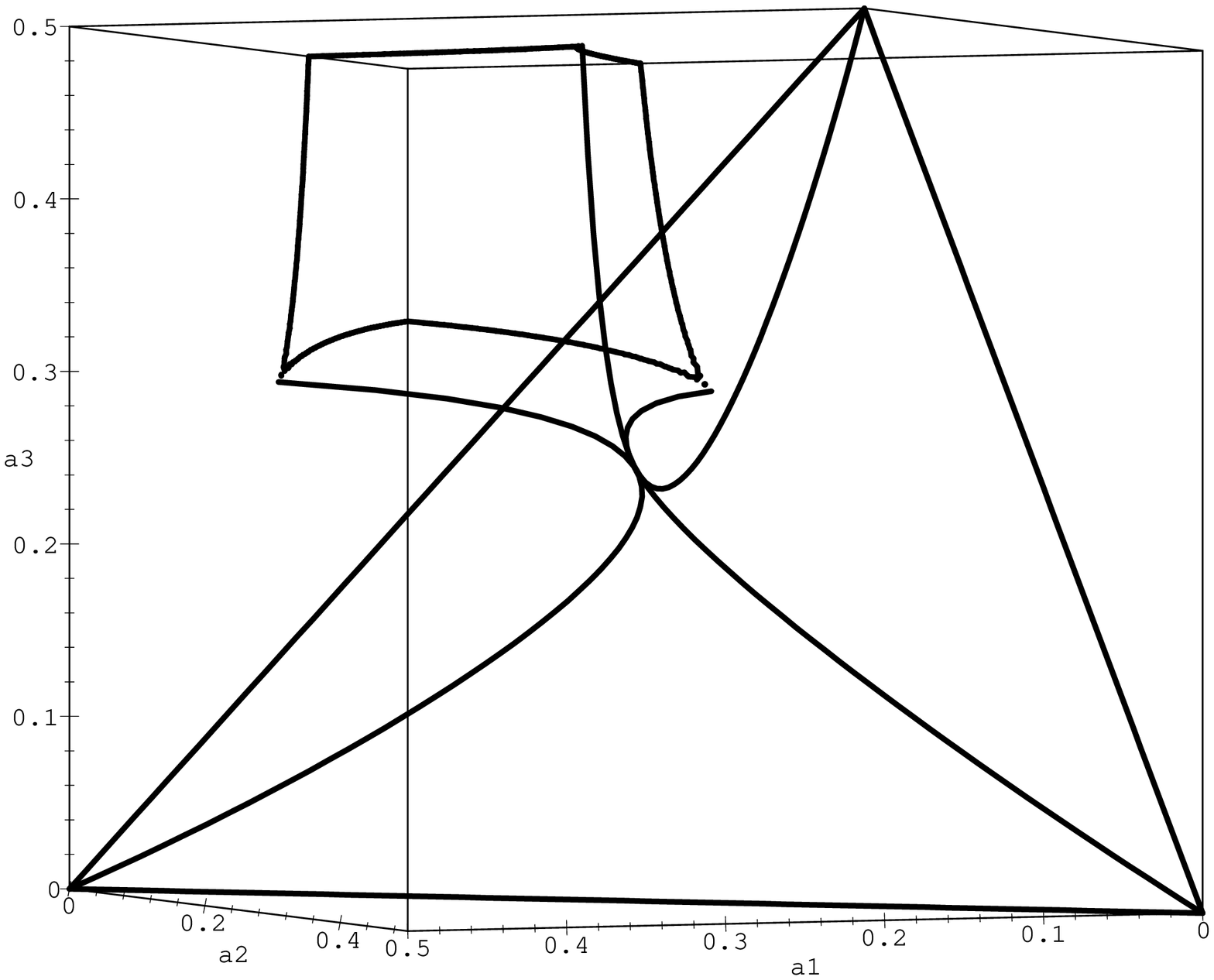}}
\caption{}
\label{singsur}
\end{figure}
\end{center}

\begin{theorem} \label{focus_center}
For $(a_1,a_2,a_3)\in O_i$ the following possibilities for singular points of the system (\ref {two_equat}) can occur:

1) If $i=1$ or $i=2$ then there is one singular point with $\delta>0$ (a node)
and three singular points with $\delta<0$ (saddles);

2) If $i=3$ then there are two singular points with $\delta<0$ (saddles).
\end{theorem}

\begin{proof}
By Theorem \ref{nonode} a non degenerate singular point
is either a node (if $\delta >0$) or a saddle (if $\delta <0$).

Recall that for $(a_1,a_2,a_3)\in O_i$ all singular points of the system (\ref {two_equat}) are not degenerate.
Moreover, there are no singular points $(x_1,x_2,x_3)$ with some zero component.
Therefore, {\bf the number of singular points}
and {\bf the set of signs} of $\delta=\widetilde{ \delta}$ for these points is constant on each  component $O_i$.
It is easy to check it for a small neighbourhood of any point
$(a_1,a_2,a_3) \in O_i$ (it follows from {\it the stability of degenerate singular points}),
but it could be spread to all points
of the connected set  $O_i$ via continuous paths (as in standard analytical monodromy theorems).

Consider the component $O_1$ containing the representative point $(1/6,1/6,1/6)$. Then as the calculations show, under $x_3=1$ the
system of equations (\ref{two_equat_sing1}) in the variables $(x_1,x_2)$ has  four solutions, given by
$(1,1)$, $(2,1)$, $(1/2,1/2)$ and  $(1,2)$.
By Lemma \ref{sing.zero} it follows that  the point  $(x_1,x_2)=(1,1)$ corresponds to the value  ${ \delta}=1/9$
(a node with  ${ \rho}=2/3$ and ${\sigma}=0$).
If  $(x_1,x_2)$  is one of the solutions $(2,1)$, $(1/2,1/2)$ and  $(1,2)$, then
${ \delta}$  equals to $-2/9$, $-8/9$ and $-2/9$  respectively (so these points are saddles).

Consider now the component $O_2$ containing the representative point $(7/15,7/15,7/15)$.  By the same manner using
Lemmas \ref{delta_rhon}
and \ref{sing.zero}  we get the  following four solutions $(x_1,x_2)$ of the system
(\ref{two_equat_sing1}):
$(1,1)$ with  ${ \delta}=169/25$ (a node with  ${ \rho}=-26/25$ and  ${\sigma}=0$),
and
$(1/14,1)$, $(1,1/14)$, $(14,14)$  with ${ \delta}$  equal to $-4901/225$,  $-4901/225$
and $-4901/44100$ respectively (three saddles).

Finally, consider the component $O_3$ containing the point $(1/6,1/4,1/3)$.
In this case we get  the following two solutions of the
system  (\ref{two_equat_sing1}):
$(x_1,x_2)=(4/5,3/5)$ with  ${ \delta}=-35/72$ (a saddle);
$(x_1,x_2) \approx (2.284185494, 2.372799295)$ with ${ \delta}=-0.0982$  (a saddle).
\end{proof}
\smallskip

Note that all points $(a_1,a_2,a_3)$ with $a_i>0$ and $a_1+a_2+a_3=1/2$ are in
$O_1$ (see discussion after the proof of Lemma \ref{sing.two}). Therefore,
by Theorem~\ref{focus_center} we get
that among the singular points determined by (\ref{solcase2}) there is one node and three saddles.
It is possible to get more detailed information in this case.
The first family of the solutions (\ref{solcase2}) corresponds to a {\it unstable} node and the other
families give saddles.

Analogous results can be obtained for all other singular points given in Section \ref{sinpsec}.
Namely we have also proved that in the case $D_1>0$ the family (\ref{solcase1.2}) can give only hyperbolic
($\lambda_1 \ne 0$, $\lambda_2 \ne 0$) or semi-hyperbolic ($\lambda_1\lambda_2=0$, $\lambda_1^2+\lambda_2^2 \ne 0$)
singular points of (\ref{two_equat}).
Note that in the hyperbolic case, nodes are stable if $\mu=1+\sqrt{D_1}$ and unstable if $\mu=1-\sqrt{D_1}$.
If $D_1=0$, $a_1=a_2=b\ne 1/4$ then family (\ref{solcase1.2}) also gives hyperbolic or semi-hyperbolic
singular points.
Note that (\ref{solcase1.2}) gives no singular points of (\ref{two_equat}) of the nilpotent type
($\lambda_1=\lambda_2=0$, $J\ne 0$, see~(\ref{linear})).

We plan to discuss results of such kind in another paper. Recall also that we have considered  in Section \ref{onespc}
the special case $D_1=0$, $b=1/4$ with a unique linearly zero ($J=0$) singular point
of (\ref{two_equat}).

\section*{Conclusion}

Theorem \ref{focus_center} gives a general picture for types of singular points of the system (\ref{two_equat})
with $(a_1,a_2,a_3)\in (0,1/2)\times (0,1/2)\times (0,1/2)$.
Nevertheless, it would be interesting to study ``degenerate'' sets of parameters $(a_1,a_2,a_3)$ from
the set $(0,1/2]^3\cap \Omega$. For the point $(a_1,a_2,a_3)=(1/4,1/4,1/4)$ we
obtained suitable results in Section \ref{onespc}.
It should be noted that the point $(a_1,a_2,a_3)=(1/4,1/4,1/4)$ is a very special one on the algebraic surface
$$
\Omega=\{(a_1,a_2,a_3)\in \mathbb{R}^3\,|\,Q(a_1,a_2,a_3)=0\}
$$
(see Section \ref{degenerset} and also Remark \ref{singdegenp}).
The following questions are worth for a further investigation.

\begin{ques}
Find a tool to study points $(a_1,a_2,a_3)$ of $\Omega$ for determining the type of singular points $(x_1,x_2)$
of the system (\ref{two_equat}).
\end{ques}

We specify this question for some special cases.

\begin{ques}
Determine the type of singular points for the case $a_i=a_j$, $i \neq j$, for $(a_1,a_2,a_3)\in \Omega$.
\end{ques}

\begin{ques}
Study the case when $(a_1,a_2,a_3)\in \Omega$ lies on one of the three curves of degenerate points of $\Omega$
($a_i=-\frac{1}{2}\frac{16t^3-4t+1}{8t^2-1}, a_j=a_k=t$, $i \neq j \neq k \neq i$).
What is the number of singular points for such $(a_1,a_2,a_3)$?
\end{ques}

\begin{ques}
Suppose that $(a_1,a_2,a_3)\in \Omega \cap (0,1/2)^3$ does not lie on the three curves of degenerate points of $\Omega$.
Then, is the number of singular points still equal to $3$?
\end{ques}

\begin{ques}
Study the case $a_k=1/2$, $a_i,a_j \in (0,1/2]$, $i \neq j \neq k \neq i$.
\end{ques}

\bigskip
{\bf Acknowledgement.}
The authors are indebted to
Prof. Yusuke~Sakane for useful discussions concerning computational aspects of this project.

\vspace{20mm}

\bibliographystyle{amsunsrt}

\begin{thebibliography}{[99]}

\bibitem{AnasChr}
{\sl Anastassiou S., Chrysikos I.} The Ricci flow approach to homogeneous Einstein
metrics on flag manifolds // J. Geom. Phys.  (2011),  V.~61, No.~8, P.~1587--1600.

\bibitem{Ar}
{\sl Arvanitoyeorgos~A.} New invariant Einstein metrics on generalized
flag manifolds //
Trans. Amer. Math. Soc. (1993), V.~337, No.~2, P.~981--995.


\bibitem{AGV}
{\sl Arnold~V.\,I., Gusein-Zade~S.\,M. and  Varchenko~A.\,N.} Singularities of differentiable
maps. Vol.~I, Monogr. Math. 82, Birkh\"{a}user Boston, Inc., Boston, MA, 1986.


\bibitem{Bes}
{\sl Besse A.\,L.}  Einstein Manifolds.  Springer-Verlag. Berlin, etc., 1987.

\bibitem{Bo}
{\sl B\"ohm C., Wilking B.}
Nonnegatively curved manifolds with finite fundamental groups admit metrics with positive Ricci curvature //
GAFA Geom. Func. Anal. (2007), V. 17, P. 665--681.

\bibitem{Bu}
{\sl Buzano M.}
Ricci flow on homogeneous spaces with two isotropy summands // Preprint,
arXiv:1209.3048 (2012).

\bibitem{ChowKnopf}
{\sl Chow~B., Knopf~D.}
The Ricci Flow: an Introduction. Mathematical Surveys and Monographs, AMS,
Providence, RI, 2004.

\bibitem{Dat}
{\sl D'Atri J.\,E., Nickerson N.} Geodesic symmetries in space with special curvature
tensors // J. Different. Geom. (1974), V.~9, P.~251--262.




\bibitem{DZ}
{\sl D'Atri J.\,E., Ziller W.} Naturally reductive metrics
and Einstein metrics on compact Lie groups // Mem. Amer.
Math. Soc. (1979), V.~215, P.~1--72.

\bibitem{Dar}
{\sl Darboux~G.} Lec\c{e}ons sur la Th\'{e}orie g\'{e}n\'{e}rale des Surfaces, IV, Gauthier-Villars, Paris,
1896.



\bibitem{Dumort}
{\sl Dumortier~F., Llibre~J., Artes~J.}
Qualitative theory of planar differential systems. Universitext. Springer-Verlag, Berlin, 2006. xvi+298 pp.

\bibitem{GP}
{\sl Glickenstein D., Payne T. L.}
Ricci flow on three-dimensional, unimodular metric Lie algebras // Preprint,
arXiv:0909.0938 (2009).



\bibitem{Ham}
{\sl Hamilton R.\,S.} Three-manifolds with positive Ricci curvature //
J. Differential Geom. (1982), V.~17, P.~255--306.

\bibitem{Isenberg}
{\sl Isenberg J., Jackson M., Lu P.} Ricci flow on locally homogeneous closed 4-manifolds //
Comm. Anal. Geom. (2006), V.~14, No.~2, P.~345--386.


\bibitem{JiangLlible}
{\sl Jiang~Q., Llibre~J.}
Qualitative classification of singular points // Qual. Theory Dyn. Syst. (2005), V.~6, No.~1, P.~87--167.


\bibitem{Kim}
{\sl Kimura~M.} Homogeneous Einstein metrics on certain K\"{a}hler C-spaces //
Adv. Stud. Pure Math. (1990), V.~18, No.~1, P.~303--320.

\bibitem{Lauret}
{\sl Lauret~J.}
Ricci flow on homogeneous manifolds // Preprint,
arXiv:1112.5900v2 (2012).


\bibitem{Lomshakov1}
{\sl Lomshakov A.\,M., Nikonorov Yu.\,G., Firsov E.\,V.} On invariant
Einstein metrics on three-locally-symmetric spaces // Doklady Mathematics (2002), V.~66, No.~2, P.~224--227.

\bibitem{Lomshakov2}
{\sl Lomshakov A.\,M., Nikonorov Yu.\,G., Firsov E.\,V.} Invariant
Einstein Metrics on Three-Locally-Symmetric Spaces // Matem. tr. (2003), V.~6, No.~2.
P.~80--101 (Russian); English translation in:
Siberian Adv. Math. (2004), V.~14, No.~3, P.~43--62.

\bibitem{Nikonorov1}
{\sl Nikonorov~Yu.\,G., Rodionov~E.\,D., Slavskii~V.\,V.}
Geometry of homogeneous Riemannian manifolds //
Journal of Mathematical Sciences (New York) (2007), V.~146, No.~7, P.~6313--6390.

\bibitem{Nikonorov2}
{\sl Nikonorov~Yu.\,G.}
On a class of homogeneous compact Einstein manifolds //
Sibirsk. Mat. Zh. (2000), V.~41, No.~1, P.~200--205 (Russian);
English translation in: Siberian Math. J. (2000), V.~41, No.~1, P.~168--172.

\bibitem{Payne}
{\sl Payne T.\,L.} The Ricci flow for nilmanifolds // J. Mod. Dyn. (2010), V.~4, No.~1, P.~65--90.

\bibitem{Rod}
{\sl Rodionov~E.\,D.}
Einstein metrics on even-dimensional homogeneous spaces
admitting a~homogeneous Riemannian metric of positive sectional curvature //
Sibirsk. Mat. Zh. (1991), V.~32, No.~3, P.126--131 (Russian);
English translation in: Siberian Math. J. (1991), V.~32, No.~3, P.~455--459.

\bibitem{Thom}
{\sl Thom~R.} Topological models in biology // Topology  (1969), V.~8, P.~313--335.

\bibitem{Topping}
{\sl Topping~P.} Lectures on the Ricci flow, London Mathematical
Society Lecture Note Series, vol. 325, Cambridge University Press, Cambridge, 2006.

\bibitem{Wal}
{\sl Wallach~N.} Compact homogeneous Riemannian manifolds with
strictly positive curvature // Annals of Mathematics, Second Series. (1972), V.~96, P.~277--295.

\bibitem{WP}
{\sl Woodcock~A.\,E.\,R., Poston~T.}
A geometrical study of the elementary catastrophes,
Lecture Notes in Mathematics, vol. 373, Springer-Verlag, Berlin-New York, 1974.


\end{thebibliography}

\vspace{5mm}

\end{document}